\documentclass[letterpaper, 10 pt, conference]{ieeeconf}  

\IEEEoverridecommandlockouts                              

\usepackage{graphicx}          
\usepackage{mwe}

\usepackage{cite}
\usepackage{amsmath,amssymb,amsfonts,mathtools}
\usepackage{textgreek}
\usepackage{graphicx}
\usepackage{textcomp}
\usepackage{caption}

\newtheorem{theorem}{Theorem}
\newtheorem{lemma}{Lemma}
\newtheorem{assumption}{Assumption}

\newtheorem{definition}{Definition}
\newtheorem{remark}{Remark}

\usepackage{times}
\usepackage{mathtools,bm}
\usepackage{dsfont}
\PassOptionsToPackage{noline,boxed}{algorithm2e}

\newcommand{\bx}{{\mathbf{x}}}
\newcommand{\bxi}{{\bm{\xi}}}
\newcommand{\bX}{{\mathit{X}}}
\newcommand{\bY}{{\mathit{Y}}}
\newcommand{\bA}{{\mathit{A}}}

\newcommand{\by}{{\mathbf{y}}}
\newcommand{\bxs}{{\mathbf{x^*}}}

\newcommand{\bz}{{\mathbf{z}}}

\newcommand{\cN}{{\mathcal N}}

\newcommand{\cG}{{\mathcal G}}
\newcommand{\cV}{{\mathcal V}}
\newcommand{\cE}{{\mathcal E}}

\newcommand{\bR}{{\mathbb R}}

\newcommand{\bE}{{\mathbb E}}

\DeclareMathOperator*{\argmin}{argmin}

\usepackage{algorithm,algpseudocode,caption}

\title{\LARGE \bf
On Linear Convergence of Distributed Stochastic Bilevel Optimization over Undirected Networks via Gradient Aggregation
}

\author{Ajay Tak$^{1}$ and Mayank Baranwal$^{2}$
\thanks{$^{1}$A.~Tak is with the Institute for Robotics and Intelligent Machines, Georgia Institute of Technology, Atlanta, GA 30332, USA
{\tt\small ajytak@gatech.edu}. $^{2}$M.~Baranwal is with the Faculty of Centre for Systems and Control, Indian Institute of Technology, Bombay, 400076, India
{\tt\small mbaranwal@iitb.ac.in}.}%
}

\begin{document}

\maketitle
\thispagestyle{empty}
\pagestyle{empty}

\begin{abstract}
Many large-scale constrained optimization problems can be formulated as bilevel distributed optimization tasks over undirected networks, where agents collaborate to minimize a global cost function while adhering to constraints, relying only on local communication and computation. In this work, we propose a distributed stochastic gradient aggregation scheme and establish its \emph{linear convergence} under the weak assumption of global strong convexity, which relaxes the common requirement of local function convexity on the objective and constraint functions. Specifically, we prove that the algorithm converges at a linear rate when the global objective function (and not each local objective function) satisfies strong-convexity. Our results significantly extend existing theoretical guarantees for distributed bilevel optimization. Additionally, we demonstrate the effectiveness of our approach through numerical experiments on distributed sensor network problems and distributed linear regression with rank-deficient data.
\end{abstract}

\section{INTRODUCTION AND RELATED WORK}\label{sec:Intro}

We study a class of bilevel optimization problem of the form~\cite{yousefian2021bilevel}:
\begin{equation}\label{eq:bilevl-cent}
    \min\limits_{\bx\in\bR^d} \sum\limits_{i=1}^nf_i(\bx) \quad \text{such that } \quad \bx\in\argmin\limits_{\bx\in\bR^d} \sum\limits_{i=1}^ng_i(\bx),
\end{equation}
where the objective functions $f_i, g_i:\bR^d\to\bR$ are assumed to be private and distributed across a network of $n$ agents. However, we assume that there is no central coordination between the agents, and since each agent has access only to its local functions $\{f_i(\cdot), g_i(\cdot)\}$, the nodes must work together to collaboratively solve the problem. 

Bilevel distributed optimization problems generalize several existing formulations in the distributed optimization literature. For example, consider the standard distributed optimization problem in the absence of strong convexity. By defining $f_i(\bx) \coloneqq \|\bx\|^2/n$, the bilevel optimization problem~\eqref{eq:bilevl-cent} reduces to finding the minimum-energy solution, i.e., the one with the smallest $\ell_2$-norm—of the classical distributed optimization problem, $\min\limits_{\bx\in\bR^d} \sum_{i=1}^{n} g_i(\bx)$. Similarly, consider the following linearly constrained distributed optimization problem:
\[
     \min\limits_{\bx\in\bR^d} \sum\limits_{i=1}^nf_i(\bx) \quad \text{s.t. } \quad \begin{array}{cl}
         A_i\bx = \mathbf{b}_i & \text{for all $i\in[n]$} \\
         \bx_j\geq 0 & \text{for all $j\in\mathcal{J}\subseteq[d]$}
     \end{array}.
\]
If the above problem is feasible, we can define constraint objective function as:
\[
    g_i(\bx) \coloneqq \frac{1}{2}\|A_i\bx-\mathbf{b}_i\|^2 + \frac{1}{2n}\sum\nolimits_{\mathcal{J}}\max\{0,-\bx_j\}^2,
\]
making it amenable to the bilevel optimization framework in~\eqref{eq:bilevl-cent}. As a result, a broad class of distributed optimization problems can be expressed within the framework of~\eqref{eq:bilevl-cent}, encompassing applications such as sensor networks, satellite tracking~\cite{hu2016smooth}, and large-scale machine learning~\cite{nathan2017optimization}, among others. The classical consensus-based approaches were initially formulated in the 1970s and 1980s~\cite{Tsitsiklis1984} with early work assuming shared global objectives distributed across networked agents. Over time, research has expanded to accommodate heterogeneous local objectives, communication constraints, and nonconvex settings~\cite{Nedic2009}.  

In the past decade, accelerated and gradient-tracking methods have emerged as key innovations for improving convergence rates in distributed settings. Algorithms such as EXTRA~\cite{Shi2015}, gradient-tracking-based approaches~\cite{pu2021distributed}, and dual-decomposition techniques~\cite{Boyd2011} have provided rigorous guarantees for convex and strongly convex problems. However, challenges remain when addressing time-varying networks and nonconvex landscapes. Recent works such as decentralized momentum methods~\cite{Scutari2020}, push-pull protocols for directed graphs~\cite{pu2020push}, and adaptive gradient schemes~\cite{han2024distributed} have demonstrated improved scalability and robustness. Most recently, \cite{yousefian2021bilevel} proposed stochastic gradient aggregation for bilevel distributed optimization over directed networks under assumptions of strong convexity on local objective functions.

Despite these advancements, existing approaches often assume strong convexity of individual objective functions or require restrictive communication topologies. Our work extends this theory by considering a bilevel distributed setting, leveraging regularization techniques to weaken convexity assumptions while preserving linear convergence rates. By incorporating the Polyak-Łojasiewicz (PL) condition, we broaden the applicability of distributed methods to a wider range of practical problems.

\noindent \textbf{Contributions}: We propose the Bilevel Distributed Aggregated Stochastic Gradient (BDASG) algorithm for solving~\eqref{eq:bilevl-cent}. We provide a rigorous theoretical analysis and establish that, under suitable assumptions the algorithm achieves linear convergence to the optimal solution. In particular, we show that the BDASG algorithm exhibits linear convergence when the global objective function, i.e., \(\sum\nolimits_{i=1}^n f_i(\bx)\), is strongly convex, without requiring individual objective functions to be convex. Moreover, we conjecture that the algorithm can be extended to guarantee linear convergence even when the global objective function satisfies the PL inequality, the weakest known condition for ensuring linear convergence. However, apart from Lemma~\ref{lem:opt}, which assumes strong convexity of the global objective function, the rest of our theoretical analysis relies only on the PL condition to establish linear convergence. To the best of our knowledge, our work provides the strongest convergence guarantees under the minimal set of assumptions. While our contributions are primarily theoretical, we validate the effectiveness of our approach through numerical experiments.

\noindent \textbf{Mathematical Notations}: We denote the set of real numbers by $\bR$. Vectors are represented using bold lowercase letters, while matrices are denoted by uppercase italicized symbols. For a function $f:\bR^d\to\bR$, the gradient at a point $\bx\in\bR^d$ is written as $\nabla f(\bx)$. Unless stated otherwise, $\|\bx\|$ refers to the Euclidean (2-norm) of $\bx$. We represent an undirected graph as $\cG \coloneqq (\cV, \cE)$ with adjacency matrix $\bA = [a_{ij}] \in \bR^{n \times n}$, where $a_{ij} \in [0,1]$. The set of nodes is given by $\cV = \{1,2,\dots,n\}$, and the 1-hop neighborhood of a node $i \in \cV$ is denoted by $\cN_i$, i.e., $\cN_i = \{j \in \cV \mid a_{ij} > 0\}$.\\ An agent $i$'s information vector is denoted by $\bx_i$, while the average information of all the agents is represented as: $\bar{\bx}\coloneqq \frac{1}{n}\sum\nolimits_{i=1}^n\bx_i$. The aggregated state is represented by $\bX\coloneqq [\bx_1 \ \bx_2 \ \dots \ \bx_n]^\intercal$ with $\bX^\dagger$ depicting the aggregated consensus error given as: $\bX^\dagger = \left(I-\frac{1}{n}\mathds{1}\mathds{1}^\intercal\right)X$, where $\mathds{1}$ is a vector of all $1$'s and $\bx_i^\dagger\coloneqq \bx_i-\bar{\bx}$ is the consensus error at the $i^{\text{th}}$-node.

\section{PRELIMINARIES AND PROBLEM FORMULATION}\label{sec:prelim}
In this section, we present preliminary results, including definitions, lemmas, and assumptions, that serve as foundational elements for the convergence analysis of the proposed algorithmic approach.
\begin{definition}[PL-inequality]\label{def:PL}
    A continuously differentiable function $q:\bR^d\to\bR$ is said to satisfy the Polyak-Łojasiewicz (PL) inequality with constant $\mu>0$ if, for all $\bx\in\bR^d$,
    \[
        \|\nabla q(\bx)\|^2 \geq 2\mu(q(\bx)-q^*),
    \]
    where $q^*$ is the minimum value of $q(\cdot)$. In addition, every strongly convex function $q(\cdot)$ with convexity modulus $\mu>0$ satisfies the PL-inequality.
\end{definition}
\begin{definition}[$L$-smooth]\label{def:L-smooth}
    A differentiable function $q:\bR^d\to\bR$ is said to be $L$-smooth if its gradient is Lipschitz continuous with constant $L$, i.e., for all $\bx,\by\in\bR^d$, the following inequality holds:
    \[
        \|\nabla q(\bx)-\nabla q(\by)\|\leq L\|\bx-\by\|.
    \]
\end{definition}
\begin{lemma}[Norm equivalence]\label{lem:norm-equiv}
    For any $\bx\in\bR^d$, the following holds: $\|\bx\|_2\leq\|\bx\|_1\leq\sqrt{d}\|\bx\|_2$.
\end{lemma}
\begin{lemma}[Co-coercivity~\cite{Vu2017ConvexOptimization}]\label{lem:coer}
   Let $h:\bR^d\to\bR$ be an $L$-smooth function. Then for all $\bx,\by\in\bR^d$, the following inequality holds:
   \[
        \langle\by-\bx,\nabla h(\by)-\nabla h(\bx)\rangle \geq \frac{1}{L}\|\nabla h(\by)-\nabla h(\bx)\|^2.
   \]
\end{lemma}

\begin{assumption}\label{ass:grad}
    Each agent $i\in\cV$ queries noisy samples of the gradients of its local functions $f_i(\cdot)$ and $g_i(\cdot)$, i.e., for some $\bz\in\bR^d$, the agent $i$ generates:
    \begin{align*}
        \tilde{\nabla} f_i(\bz) &= \nabla f_i(\bz) + \bxi_{f_i},\\
        \tilde{\nabla} g_i(\bz) &= \nabla g_i(\bz) + \bxi_{g_i},
    \end{align*}
    where $\bxi_{f_i}$ and $\bxi_{g_i}$ are independent random vectors in $\bR^d$ with zero mean, i.e., $\mathbb{E}[\bxi_{f_i}]=\mathbb{E}[\bxi_{g_i}]=\mathbf{0}$.
\end{assumption}
\begin{remark}\label{rem:stochastic}
    Assumption~\ref{ass:grad} implies that $\tilde{\nabla} f_i(\cdot)$ and $\tilde{\nabla} g_i(\cdot)$ are unbiased estimates of $\nabla f_i(\cdot)$ and $\nabla g_i(\cdot)$, respectively. In addition, we assume bounded noise in stochastic gradients, i.e., without loss of generality with probability 1, the following holds:
    \[
        \|\bxi_{f_i}\|\leq C_f, \qquad \|\bxi_{g_i}\|\leq C_g, \quad \text{for all} \ i\in\cV.
    \]
\end{remark}
\begin{assumption}\label{ass:funcs}
    The cumulative objective function $f(\bx)\coloneqq \sum_{i=1}^nf_i(\bx)$ is strongly convex with $\mu>0$, while the cumulative constraint function $g(\bx)\coloneqq \sum_{i=1}^ng_i(\bx)$ is convex.
\end{assumption}
\begin{remark}\label{rem:g-plus-f}
    Assumption~\ref{ass:funcs} only requires the sum of the private functions to be (strongly) convex, allowing for the possibility that the individual functions may be non-convex. Moreover, Assumption~\ref{ass:funcs} also implies that for some $\lambda > 0$, the function $g(\bx) + \lambda f(\bx)$ is strongly convex, with a modulus of convexity $\mu_\lambda\coloneqq\mu\lambda > 0$, and that the optimal solution of~\eqref{eq:bilevl-cent} is unique if it exists. Moreover, from Assumption~\ref{ass:grad}, since the noise is bounded, we define $C\coloneqq \max\limits_{i\in cV}\|\bxi_i\|$, where $\bxi_i\coloneqq \bxi_{g_i}+\lambda\bxi_{f_i}$.
\end{remark}
\begin{assumption}\label{ass:smoothness}
    The local functions $\{f_i\}$ and $\{g_i\}$ are $\{L_{f_i}\}$ and $\{L_{g_i}\}$-smooth, respectively.
\end{assumption}
\begin{remark}\label{rem:L}
    Assumption~\ref{ass:smoothness} implies that the function $g_i(\bx)+\lambda f_i(\bx)$ is $L_i\coloneqq L_{g_i}+\lambda L_{f_i}$-smooth. In addition, the function $g(\bx) + \lambda f(\bx)$ defined in Remark~\ref{rem:g-plus-f} is $\bar{L}\coloneqq \sum\nolimits_{i\in\cV} L_{i}$-smooth.
\end{remark}

\begin{lemma}\label{lem:opt}
    Assume Assumption~\ref{ass:funcs} holds. Let $\bx^*$ denote the unique optimizer of \eqref{eq:bilevl-cent}, and let $\bx_\lambda^*$ represent the minimizer of $g(\cdot)+\lambda f(\cdot)$. Then, the norm of the difference, $\|\bx^* - \bx_\lambda^*\|$, converges to zero as $\lambda \to 0$.
\end{lemma}
\begin{proof}
    Note that the bilevel optimization problem in \eqref{eq:bilevl-cent} can be equivalently expressed as $\min\limits_{\by\in\argmin g(\bx)}f(\by)$. Thus, if $\bx^*$ is the unique optimizer of \eqref{eq:bilevl-cent}, then $g(\bx^*)\leq g(\bx)$ for all $\bx$. Since $g+\lambda f$ is $\mu\lambda$-strongly convex, it follows that
    \begin{align}\label{eq:g-plus-f}
        g(\bx^*) + \lambda f(\bx^*) \geq g(\bx_\lambda^*) + \lambda f(\bx_\lambda^*) + \frac{\mu\lambda}{2}\|\bx^*-\bx_\lambda^*\|^2,
    \end{align}
    which implies that in the limit $\lambda\to0$, $g(\bx^*)\geq g(\lim\limits_{\lambda\to0}\bx_\lambda^*)$ indicating that $\hat{\bx}\coloneqq\lim\limits_{\lambda\to0}\bx_\lambda^*$ is also a feasible solution to \eqref{eq:bilevl-cent}, with $\hat{\bx}$ also being a minimizer of $g(\bx)$. Taking the limit $\lambda\to0$ in \eqref{eq:g-plus-f} and using the continuity of $f$, we obtain that $f(\bx^*)\geq f(\hat{\bx})$. From the uniqueness of $\bx^*$, it must follow that $\lim\limits_{\lambda\to0}\bx_\lambda^*=\bx^*$.
\end{proof}
\begin{remark}
    Lemma~\ref{lem:opt} implies that for sufficiently small values of $\lambda$, the optimal solution $\bx_\lambda^*$ remains close to the unique optimal solution of~\eqref{eq:bilevl-cent}. Henceforth, with a slight abuse of notation, we denote $\bx_\lambda^*$ by $\bx^*$, assuming $\lambda$ is sufficiently small.
\end{remark}

\section{MAIN RESULTS}\label{sec:mainres}
We introduce the Bilevel Distributed Aggregated Stochastic Gradient (BDASG) algorithm, which achieves linear convergence under suitable assumptions. At each iteration, the algorithm aggregates stochastic gradients and integrates consensus with gradient updates to converge to the optimal solution. The detailed steps are outlined in Algorithm~\ref{alg:BDASG}. To this end, we introduce the following notation:
\begin{align}\label{eq:b-def}
    h_i(\bx_i(k))&\coloneqq\!\tilde{\nabla}g_i(\bx_i(k))\!+\!\lambda \tilde{\nabla}f_i(\bx_i(k)) \nonumber\\
    &= \underbrace{\nabla g_i(\bx_i(k))\!+\!\lambda\nabla f_i(\bx_i(k))}_{\coloneqq\nabla b_i(\bx_i(k))} + \underbrace{\bxi_{g_i}\!+\!\lambda\bxi_{f_i}}_{\bxi_i}
\end{align}

\begin{algorithm}
    \caption{Bilevel Distributed Aggregated Stochastic Gradient (BDASG) algorithm}
    \begin{algorithmic}[1]
    \State \textbf{Initialize}: Each agent $i\in\cV$ initializes at an arbitrary state $\bx_i(0)$ and $\by_i(0)=\tilde{\nabla}g_i(\bx_i(0))+\lambda\tilde{\nabla}f_i(\bx_i(0))$.
        \While{$k = 0, 1, 2, \dots$}\Comment{Iterate until convergence}
            \State $\bx_i(k\!+\!1) \gets \sum\limits_{j\in\cN_i}a_{ij}\bx_j(k)-\alpha\by_i(k)$
            \State $\by_i(k\!+\!1) \gets \sum\limits_{j\in\cN_i}a_{ij}\!\by_j(k)\!+\!h_i(\bx_i(k\!+\!1))\!-\!h_i(\bx_i(k))$
        \EndWhile
    \end{algorithmic}
    \label{alg:BDASG}
\end{algorithm}

Algorithm~\ref{alg:BDASG} can be summarized as the following set of iterative updates:
\begin{align}\label{eq:alg}
    \bX(k+1) &= A\bX(k) - \alpha\bY(k),\nonumber\\
    \bY(k+1) &= A\bY(k) + H(\bX(k+1)) - H(\bX(k)),
\end{align}
where $H(\bX(k))\coloneqq \tilde{\nabla}G(\bX(k)) + \lambda\tilde{\nabla}F(\bX(k))$, and $k$ is the iteration number. Here, the functions $G(\bX(k))$ and $F(\bX(k))$ are defined as:
\begin{equation*}
    G(\bX(k)) = \left[\begin{array}{c}
         g_1(\bx_1(k)) \\
         g_2(\bx_2(k)) \\
         \vdots \\
         g_n(\bx_n(k))
    \end{array}\right]\!\!, F(\bX(k)) = \left[\begin{array}{c}
         f_1(\bx_1(k)) \\
         f_2(\bx_2(k)) \\
         \vdots \\
         f_n(\bx_n(k))
    \end{array}\right]\!.
\end{equation*}
The averaged dynamics follows:
\begin{align}\label{eq:avg-dynamics}
    \bar{\bx}(k+1) &= \bar{\bx}(k) - \alpha \bar{\by}(k), \nonumber \\
    \bar{\by}(k+1) &= \bar{H}(\bX(k+1)) \coloneqq \frac{1}{n}\sum\nolimits_{i=1}^nh_i(\bx_i(k+1)),
\end{align}
where $h_i(\bx_i)\coloneqq\tilde{\nabla}g_i(\bx_i)+\lambda\tilde{\nabla}f_i(\bx_i)$.\\
\\
Now, that all the definitions and assumptions have been stated, we look at a series of lemmas and their proofs which are essential in the analysis of our algorithm.
This section analyzes the convergence rate of BDASG, showing that it achieves linear convergence in expectation to a neighborhood of the solution of problem~\eqref{eq:bilevl-cent}. Inspired by~\cite{doan2018aggregating}, our analysis consists of two steps: first, we establish that the nodes in the graph reach consensus at a linear rate, meaning the difference between each node's value \(\bx_i\) and the average state \(\bar{\bx}\) diminishes linearly, i.e., \(\|\bx_i - \bar{\bx}\| \to 0\) at a linear rate (Lemma\eqref{lem:X-dagger}); second, we show that the average state \(\bar{\bx}\) itself converges linearly to the optimal solution \(\bxs\) (Theorem~\ref{thm:avg-red}), demonstrating that BDASG achieves linear convergence to the optimal solution. Before we begin our analysis, we will first state an assumption, the proof to which will be presented at the end. \\
\\
\textbf{Gradient Reduction at a Linear Rate}: Given some positive constants \(B\), \(D\), and \(\gamma \in (\sigma_2, 1)\), the sequence
\(\{y(k)\}\) generated by Algorithm 1 satisfies
\begin{equation} \label{eq:y_grad_red}
    \mathbb{E}[\|Y^{\dagger}(k)\|] \leq D \gamma^k + B, \quad  k \geq 0,
\end{equation}
where $\sigma_2$ is the second largest singular value of the graph adjacency matrix. We will prove this statement later in the analysis. We begin by establishing the linear convergence of \(\|\textbf{x}_i - \bar{\textbf{x}}\| \to 0\) under condition \eqref{eq:y_grad_red}. To this end, we first present a series of supporting lemmas, followed by the proof of the main result in Theorem~\ref{thm:avg-red}.
\begin{lemma}\label{lem:X-dagger}
    Let $\{\bx_i(k), \by_i(k)\}$ be the sequence generated by Algorithm~\ref{alg:BDASG}, and let $\sigma_2$ denote the second largest singular value of the graph adjacency matrix $A$. Then the norm of the overall consensus error satisfies the following inequality:
    \[
        \bE[\|\bX^\dagger(k)\|] \leq \sigma_2^k\bE[\|\bX^\dagger(0)\|] + \alpha\sum\limits_{t=0}^{k-1}\sigma_2^{k-1-t}\bE[\|\bY^\dagger(t)\|]
    \]
\end{lemma}
\begin{proof}
    The proof follows by iteratively expanding $\bX^\dagger(k)$ as:
    \begin{align*}
        X^\dagger(k) &=  \left(I\!-\!\frac{1}{n} \mathbf{1} \mathbf{1}^\intercal \right)\!\{\!AX(k\!-\!1)-\alpha Y(k\!-\!1)\!\} \\
        &= A X^{\dagger} (k-1) - \alpha Y^\dagger(k-1) \\
        &= A^k X^\dagger(0) - \alpha \sum\nolimits_{t=0}^{k-1} A^{k-1-t} Y^\dagger(t) \\
        \implies\!\!\!\bE[\| X^\dagger(k) \|] &\leq \!\sigma_2^k \bE[\| X^\dagger(0) \|] \!+\!\alpha \!\!\sum_{t=0}^{k-1}\!\!\sigma_2^{k\!-\!1\!-\!t}\bE[\| Y^\dagger(t) \|],
    \end{align*}
    which completes the proof.
\end{proof}
Assuming the linear gradient reduction property in \eqref{eq:y_grad_red} (as established in Theorem~\ref{thm:grad-red}) and leveraging Lemma~\ref{lem:X-dagger}, we demonstrate that the sequence generated by Algorithm~\ref{alg:BDASG} exhibits linear convergence of \(\|\bx_i - \bar{\bx}\| \rightarrow 0\).\\
By \eqref{eq:y_grad_red} and since \(\gamma \in (\sigma_2,1)\), the second term on the right-hand side of Lemma~\ref{lem:X-dagger} is given as:
\begin{align*}
    \sum_{t=0}^{k-1} \sigma_2^{k-1-t} \mathbb{E} \left[ \|Y^{\dagger} (t) \|\right] &\leq \sum_{t=0}^{k-1} \sigma_2^{k-1-t} \left(D \gamma^t + B\right) \\
    \implies\sum_{t=0}^{k-1} \sigma_2^{k-1-t} \mathbb{E} \left[ \|Y^{\dagger} (t) \|\right] &\leq \frac{D \gamma^k}{\gamma - \sigma_2} + \frac{B}{1 - \sigma_2}.   
\end{align*}
Therefore
\begin{align}\label{eq:consensus}
    \mathbb{E} \left[ \|X^{\dagger} (k) \|\right] \leq 
\left( \frac{\mathbb{E} \left[ \|X^{\dagger} (0) \|\right] + D \alpha}{\gamma - \sigma_2} \right)\!\gamma^k 
+ \frac{B \alpha}{1\!-\!\sigma_2}.
\end{align}
Now that we have established the linear convergence of consensus, we state some additional results which will help us in proving that \(\bar{\bx}\) converges to \(\bx^*\) linearly.
\begin{lemma}\label{lem:gamma}
    Let Assumptions~\ref{ass:grad}-\ref{ass:smoothness} hold. Let $\{\bx_i(k), \by_i(k)\}$ be the sequence generated by Algorithm~\ref{alg:BDASG} for all $i\in\cV$. Then the averaged state vector $\bar{\bx}(k)$ satisfies:
    \[
        \left\|\bar{\bx}(k)-\bx^*-\frac{\alpha}{n}\sum\nolimits_{i=1}^n\nabla b_i(\bar{\bx}(k))\right\| \leq \theta\|\bar{\bx}(k)-\bx^*\|,
    \]
    where $\theta=\sqrt{1-\frac{\alpha\mu_\lambda^2}{n}\left(\frac{2}{\bar{L}}-\frac{\alpha}{n}\right)}\in (0,1)$, and $\nabla b_i(\bx(k))$ is defined in~\eqref{eq:b-def}.
\end{lemma}
\begin{proof} Recall that
\begin{align}\label{eq:p1}
    &\left\| \bar{\bx}(k) - \bx^* - \frac{\alpha}{n} \sum\nolimits_{i=1}^n \nabla b_i(\bar{\bx}(k)) \right\|^2 \nonumber\\
    &=\| \bar{\bx}(k) - \bx^* \|^2 + \frac{\alpha^2}{n^2} \left\|\sum\nolimits_{i=1}^n \nabla b_i(\bar{\bx}(k)) \right\|^2 \nonumber \\
    &\ \ \ \ -  \frac{2\alpha}{n} \left\langle \bar{\bx}(k) - \bx^*, \sum\nolimits_{i=1}^n \nabla b_i(\bar{\bx}(k)) \right\rangle \nonumber\\
   & \leq \| \bar{\bx}(k) - \bx^* \|^2 + \frac{\alpha^2}{n^2} \left\|\sum\nolimits_{i=1}^n \nabla b_i(\bar{\bx}(k)) \right\|^2 \nonumber \\
   &\ \ \ -  \frac{2\alpha}{n\bar L} \left\|\sum\nolimits_{i=1}^n \nabla b_i(\bar{\bx}(k)) \right\|^2 \ (\text{From Lemma~\ref{lem:coer}}) \nonumber\\
   &= \| \bar{\bx}(k) - \bx^* \|^2 -\frac{\alpha}{n} \left( \frac{2}{\bar L} - \frac{\alpha}{n} \right) \!\!\!\!\!\!\!\!\!\!\!\!\!\!\!\underbrace{\left\|\sum\nolimits_{i=1}^n \nabla b_i(\bar{\bx}(k)) \right\|^2}_{\leq 2\mu_\lambda\left(\sum_{i=1}^n b_i(\bar{\bx}(k))-\sum_{i=1}^n b_i(\bar{\bx}^*)\right)}\!\!\!\!\!\!\!\!\!\!\!\!\!\!\!,
\end{align}
where the last step follows from applying PL-inequality to the function $b(\bx)=g(x)+\lambda f(x)$. Recall that PL-functions exhibit at least quadratic growth resulting in the following upper-bound on~\eqref{eq:p1} as:
\begin{align*}
    \eqref{eq:p1}&\leq \| \bar{\bx}(k) - \bx^* \|^2 -\frac{2 \mu_\lambda\alpha}{n} \left( \frac{2}{\bar L} - \frac{\alpha}{n} \right) \frac{\mu_\lambda}{2} \| \bar{\bx}(k) - \bx^* \|^2 \\
    &= \left(1 -\frac{ \mu_\lambda^2 \alpha}{n} \left( \frac{2}{\bar L} - \frac{\alpha}{n} \right)\right) \| \bar{\bx}(k) - \bx^* \|^2
\end{align*}
Defining $\theta^2 = 1 - \frac{\mu_{\lambda}^2 \alpha}{n} \left(  \frac{2}{\bar L} - \frac{\alpha}{n} \right)$, and considering  \(0 < \frac{\alpha}{n} < \frac{2}{\bar L}\), we have  \(\theta^2 < 1\). In order to complete the proof, we still need to show that \(\theta^2 > 0\). Observe that
\begin{align*}
    \theta^2 &= 1 - \frac{\mu_\lambda^2 \alpha}{n} \left(  \frac{2}{\bar L} - \frac{\alpha}{n} \right) = \frac{\alpha^2}{n^2} \mu_\lambda^2 - 2 \mu_\lambda^2 \frac{2}{\bar L} \frac{\alpha}{n} + 1 \\
    \implies \theta^2  &= \left( \frac{\alpha \mu_{\lambda}}{n}- \frac{\mu_{\lambda}}{\bar{L}} \right)^2 + \left( 1 - \frac{\mu_{\lambda}^2}{\bar{L}^2} \right) \geq 0,
\end{align*}
implying that $\theta = \sqrt{1 - \frac{\mu_\lambda^2 \alpha}{n} \left(\frac{2}{\bar L} - \frac{\alpha}{n} \right)} \in (0,1)$, which completes the proof.
\end{proof}
Using the Lemma stated above, we now prove linear convergence of \(\bar{\bx}\) to \(\bx^*\) , which will complete our analysis of the algorithm.
\begin{theorem}\label{thm:avg-red}
    Let Assumptions~\ref{ass:grad}-\ref{ass:smoothness} hold. Let $\{\bx_i(k), \by_i(k)\}$ be the sequence generated by Algorithm~\ref{alg:BDASG} for all $i\in\cV$. Given $\gamma\in(\sigma_2,1)$, and let the step size $\alpha$ be chosen as
    \[
        \frac{n}{\bar{L}}+n\sqrt{\frac{(\sigma_2+\tau)^2-1}{\mu_\lambda}\!+\!\frac{1}{\bar{L}^2}}\leq\alpha\leq\frac{n}{\bar{L}}+n\sqrt{\frac{\gamma^2\!-\!1}{\mu_\lambda}\!+\!\frac{1}{\bar{L}^2}},
    \]
    where $\tau>0$ is a tuning parameter and $\sigma_2$ is the second largest singular value of the adjacency matrix $A$. Then the averaged state vector $\bar{\bx}(k)$ converges linearly to the neighborhood of $\bx^*$ in expectation, i.e.,
    \begin{align*}
        &\bE[\|\bar{\bx}(k+1)-\bx^*\|] \leq\Bigg(\bE[\|\bar{\bx}(0)-\bx^*\|]+ \\
        &\frac{\tilde{L}\sqrt{d}\alpha\bE[\|\bX^\dagger(0)\|]}{n(\gamma\!-\!\sigma_2)}\!\!\Bigg)\!\gamma^{k+1}\!\!+\!\!\frac{c\alpha}{n(1\!-\!\gamma)}\!\!+\!\frac{\tilde{L}\sqrt{d}\alpha^2}{n(1\!-\!\gamma)}\!\!\left(\!\!\frac{B}{\gamma\!-\!\sigma_2}\!+\!\frac{D}{\gamma}\!\right)
    \end{align*}
\end{theorem}
\begin{proof}
    From~\eqref{eq:avg-dynamics} and using $\bar{\by}(k) = \frac{1}{n} \sum\limits_{i=1}^{n}{h}_i(\bx_i(k))$ we have
\begin{align*}
 \|&\bar{\bx}(k+1) - \bx^*\| =\left\|\bar{\bx}(k) - \bx^* - \frac{\alpha}{n} \sum_{i=1}^{n} h_i(\bx_i(k))\right\| \\
    &= \left\| \bar{\bx}(k) - \bx^* - \frac{\alpha}{n} \sum_{i=1}^{n} h_i(\bx_i(k))- \frac{\alpha}{n} \sum_{i=1}^{n} \nabla b_i(\bar{\bx}(k)) \right. \nonumber \\
    &\ \ \ \ \left. +\frac{\alpha}{n} \sum_{i=1}^{n} \nabla b_i(\bar{\bx}(k))\right\| \\
    &\leq \left\| \bar{\bx}(k) - \bx^* - \frac{\alpha}{n} \sum_{i=1}^{n} \nabla g_i(\bar{\bx}(k))+\lambda \nabla f_i(\bar{\bx}(k))\right\|\\
    &\ \ \ \ +\frac{\alpha}{n} \sum_{i=1}^{n} \|h_i(\bx_i(k)))-\nabla g_i(\bar{\bx}(k))-\lambda \nabla f_i(\bar{\bx}(k))\| \\
    &\leq \left\| \bar{\bx}(k) - \bx^* - \frac{\alpha}{n} \sum_{i=1}^{n} \nabla b_i(\bar{\bx}(k))\right\| \nonumber \\
    &\ \ \ \ +\frac{\alpha}{n} \sum_{i=1}^{n} \|\nabla b_i(\bx_i(k)))-\nabla b_i(\bar{\bx}(k))\| + \|\boldsymbol{\xi}_i^k\| \\
    &\stackrel{\text{(Lemma~\ref{lem:gamma})}}{\!\!\!\!\!\!\!\!\!\!\!\!\!\!\!\!\!\!\!\!\!\leq} \!\!\!\!\!\!\!\!\!\!\!\!\!\!\!\!\theta \|\bar{\bx}(k)-\bx^*\|_2 + \frac{\alpha}{n} \sum\nolimits_{i=1}^{n} L_{i} \|\bx_i(k)-\bar{\bx}(k)\|_2 + \frac{\alpha C}{n}\\
    &\leq \theta \|\bar{\bx}(k)-\bx^*\|_2 + \frac{\alpha}{n} \sum\nolimits_{i=1}^{n} L_{i} \|\bx_i(k)-\bar{\bx}(k)\|_1 + \frac{\alpha C}{n}\\
    &\leq \theta \|\bar{\bx}(k)-\bx^*\|_2 + \frac{\alpha C}{n} + \frac{\alpha}{n} \tilde{L}  \|X(k)-\mathbf{1} \bar{\bx}^\intercal\|_1 \\
    &= \theta \|\bar{\bx}(k)-\bx^*\|_2 + \frac{\alpha C}{n} + \frac{\alpha \tilde{L}}{n}  \|X^\dagger(k)\|_1 \\
    &\leq \theta \|\bar{\bx}(k)-\bx^*\|_2 + \frac{\alpha C}{n} + \frac{\alpha \tilde{L} \sqrt{d}}{n}  \|X^\dagger(k)\|,
\end{align*}
where the last inequality follows from norm equivalence. Observe that given the bounds on \(\alpha\), we have $\sigma_2 \leq \theta = \sqrt{1 - \frac{\mu^2 \alpha}{n} \left(  \frac{2}{L} - \frac{\alpha}{n} \right)} \leq \gamma \leq 1$. Substituting $\|\bX^\dagger\|$ from Lemma~\ref{lem:X-dagger} into the above equation and considering linear gradient reduction~\eqref{eq:y_grad_red} yields,

\begin{align*}
    &\mathbb{E} \left[\|\bar{\bx}(k+1) - \bx^*\|\right] 
    \leq \mathbb{E} \left[\|\bar{\bx}(0) - \bx^*\|\right] \!\gamma^{k+1} + \frac{C\alpha}{n(1\!-\!\gamma)} \nonumber \\
    &\ \ \ \ \ \ \ \ \ \ \ \ \ \ \ \ \ \ \ \ + \frac{\bar{L}\sqrt{d}\alpha \mathbb{E} \left[\|X^\dagger(0)\|\right]}{n(\gamma - \sigma_2)} \gamma^{k+1} \\
    &\ \ \ \ \ \ \ \ \ \ \ \ \ \ \ \ \ \ \ \ + \frac{\bar{L}\alpha^2\sqrt{d}}{n} \sum_{s=0}^{k} \gamma^{k-s} \sum_{t=0}^{s-1} \sigma_2^{s-1-t} (D \gamma^t + B),
\end{align*}
which upon further simplification reduces to:

\begin{align*}
    &\mathbb{E} \left[\|\bar{\bx}(k\!+\!1) \!-\! \bx^*\|\right] \\
    & \ \ \leq \left(\mathbb{E}\left[\|\bar{\bx}(0) - \bx^*\|\right]+\frac{\bar{L}\sqrt{d}\alpha \mathbb{E}\left[\|X^\dagger(0)\|\right]}{n(\gamma - \sigma_2)}\right) \gamma^{k+1} \\
    &\ \ + \frac{C\alpha}{n(1-\gamma)}+\frac{\bar{L}\sqrt{d}\alpha^2}{n(1-\gamma)}\left(\frac{B}{\gamma-\sigma_2}+\frac{D}{\gamma}\right).
\end{align*}
\end{proof}
As observed in the above result, the first term on the right-hand side decays linearly to zero. In contrast, the second and third terms remain dependent on \(\alpha\). A similar pattern is evident in~\eqref{eq:consensus}. Consequently, for any desired accuracy level, one can choose a sufficiently small \(\epsilon\) such that \(\mathbb{E}[\|\bx_i(k) - \bx^*\|] \leq \epsilon\) for sufficiently large \(k\).  This concludes our analysis of the linear convergence of Algorithm~\ref{alg:BDASG}. Next, we establish the gradient reduction at a linear rate property~\eqref{eq:y_grad_red}, which plays a crucial role in proving the linear convergence of our proposed algorithm.

We now establish a series of lemmas, leading to the proof of the main theorem, which demonstrates that Algorithm~\ref{alg:BDASG} indeed achieves linear gradient reduction.
\begin{lemma}\label{lem:H}
    Let $\{\bx_i(k), \by_i(k)\}$ be the sequence generated by Algorithm~\ref{alg:BDASG}. Then there exist constants $\beta_1, \beta_2>0$ such that the vector-valued function $H(\cdot)$ defined in~\eqref{eq:alg} satisfies the following inequality:
    \[
        \|H(\bX(k+1))-H(\bX(k))\|_2 \leq \beta_1\|\bX(k+1)-\bX(k)\|_2 + \beta_2.
    \]
\end{lemma}
\begin{proof}
Recall that from norm equivalence:

\begin{align*}
    &\|H(\!\bX(k\!+\!1))\!-\!H(\!\bX(k))\|_2 \leq\! \|H(\!\bX(k\!+\!1))\!-\!H(\!\bX(k))\|_1 \\
    &\ \ \ \ \leq \sqrt{d}\sum_{i=1}^n \|h_i(\bx_i(k+1))\!-\!h_i(\bx_i(k))\|_2 \\
    & \ \ \ \ \leq\sqrt{d}\sum_{i=1}^n \left(\| \nabla b_i(\bx_i(k+1)) - \nabla b_i(\bx_i(k)) \|_2 + 2C\right).
\end{align*}
Using Lipschitz-smoothness of $b_i(\cdot)$, the above inequality can be further simplified to:

\begin{align*}
    &\|H(\bX(k+1))\!-\!H(\bX(k))\|_2 \\
    &\leq \sqrt{d} \bar{L} \|X(k+1)-X(k) \|_1 +2C\sqrt{d} \\
    &\leq \underbrace{d \bar{L}}_{\beta_1}\|X(k+1)-X(k) \|_2 +\underbrace{2C\sqrt{d}}_{\beta_2},
\end{align*}
which completes the proof.
\end{proof}

\begin{lemma}\label{lem:app1}
    Let Assumptions~\ref{ass:grad}-\ref{ass:smoothness} hold. Consider the sequence $\{\bx_i(k), \by_i(k)\}$ generated by Algorithm~\ref{alg:BDASG} for all $i \in \cV$. Given $\gamma \in (\sigma_2,1)$, let the step size $\alpha$ be chosen as  
    \[
        \frac{n}{\bar{L}} + n \sqrt{\frac{(\sigma_2\!+\! \tau)^2 - 1}{\mu_\lambda} + \frac{1}{\bar{L}^2}} \leq \alpha \leq \frac{n}{\bar{L}} + n \sqrt{\frac{\gamma^2\!-\!1}{\mu_\lambda} + \frac{1}{\bar{L}^2}},
    \]
    where $\tau > 0$ is a tuning parameter. Then, the following bound holds:  
    \begin{align*}
        \mathbb{E}[\|\bar{\bx}(k+1) - \bx^*\|] &\leq \beta_3 \theta^{k+1} + \frac{\beta_1}{1 - \theta} \\
        &\ \ \ \ +\frac{\alpha \beta_2 \theta}{\sigma_2 (\theta - \sigma_2)} \sum_{l=0}^{k-1} \sigma_2^{k-l} \mathbb{E}[\|Y^\dagger(l)\|],
    \end{align*}
    where $\beta_3 > 0$ is a constant, and $\theta$ is as defined in Lemma~\ref{lem:gamma}.
\end{lemma}

\begin{proof} Applying triangle inequality to the averaged dynamics yields:
\begin{align*}
    \|\bar{\bx}(k+1) - \bxs\| \leq  \left\|\bar{\bx}(k)\!-\!\bxs\!-\!\frac{\alpha}{n}\!\sum_{i=1}^{n} \nabla b_i(\bar{\bx}(k))\right\|\\ + \frac{\alpha}{n}\!\sum_{i=1}^{n} \|h_i(\bx_i(k))\!-\!\nabla b_i(\bar{\bx}(k))\|,
\end{align*}
which can be further simplified using Lemma~\ref{lem:gamma} as:

\begin{align*}
    &\|\bar{\bx}(k+1) - \bxs\| \leq \theta \|\bar{\bx}(k)-\bx^*\| \\
    &\ \ \ \ \ \ \ \ \ + \frac{\alpha}{n} \sum_{i=1}^{n} \|\nabla b_i(\bx_i(k)))-\nabla b_i(\bar{\bx}(k))\| + \|\boldsymbol{\xi}_i^k\| \\
    &\ \ \ \ \ \ \ \ \ \leq \theta \|\bar{\bx}(k)-\bx^*\| + \frac{\alpha C}{n} + \frac{\alpha}{n} \bar{L} \underbrace{\|X(k)-\mathbf{1} \bar{\bx}^\intercal\|_1}_{\|\bX^\dagger(k)\|} \\
    &\ \ \ \ \ \ \ \ \ = \theta \|\bar{\bx}(k)-\bx^*\| + \frac{\alpha C}{n} + \frac{\alpha\bar{L}\sqrt{d}}{n}\|\bX^\dagger(k)\|.
\end{align*}
Iteratively expanding the above equation yields:

\begin{align*}
    \|&\bar{\bx}(k+1) - \bxs\| \leq \theta^{k+1}\|\bar{\bx}(0)-\bxs\| + \frac{\beta_1}{1\!-\!\theta}\\
    &\ \ \ \ \ + \beta_2\!\sum_{t=0}^k\theta^{k-t}\!\left(\!\sigma_2^t\|X^\dagger(0)\|+\alpha\sum_{l=0}^{t-1}\sigma_2^{t-1-l}\|Y^\dagger(l)\|\!\right)
\end{align*}
Rearranging the terms, the above inequality reads:
\begin{align*}
    &\|\bar{\bx}(k+1) - \bxs\| \leq \left(\|\bar{\bx}(0)-\bx^*\|+\frac{\beta_2 \|X^\dagger(0)\|}{\theta-\sigma_2}\right)\theta^{k+1} \\
    & \ \ \ \ \ \ \ \ \ + \frac{\beta_1}{1-\theta} + \alpha\beta_2 \sum_{l=0}^{k-1}\sigma_2^{k-1-l}\|Y^\dagger(l)\| \sum_{t=l+1}^{k}\left(\frac{\sigma_2}{\theta}\right)^{t-k} \\
    &\leq \beta_3 \theta^{k+1}+\frac{\beta_1}{1-\theta}+\frac{\alpha \beta_2 \theta}{\sigma_2 \left(\theta-\sigma_2\right)} \sum_{l=0}^{k-1}\sigma_2^{k-l}\|Y^\dagger(l)\|,
\end{align*}
which completes the proof.
\end{proof}

\begin{lemma}\label{lem:app2}
    Let Assumptions~\ref{ass:grad}-\ref{ass:smoothness} hold. Let $\{\bx_i(k), \by_i(k)\}$ be the sequence generated by Algorithm~\ref{alg:BDASG} for all $i\in\cV$. Given $\gamma\in(\sigma_2,1)$, and let the step size $\alpha$ be chosen as
    \[
        \frac{n}{\bar{L}}+n\sqrt{\frac{(\sigma_2+\tau)^2-1}{\mu}\!+\!\frac{1}{\bar{L}^2}}\leq\alpha\leq\frac{n}{\bar{L}}+n\sqrt{\frac{\gamma^2\!-\!1}{\mu}\!+\!\frac{1}{\bar{L}^2}},
    \]
    where $\tau>0$ is a tuning parameter. Then, the following bound holds:
    \begin{align*}
        \mathbb{E}[\|X(k\!+\!1)\!-\!X(k)\|_2] \leq \beta_4 \theta^k \!+\! \alpha \mathbb{E}[\|Y^\dagger(k)\|]\!+\!\frac{2n\beta_1\!\sqrt{d}}{1-\theta} \\ + \beta_5 \sum_{t=0}^{k-1}\!\theta^{k\!-\!1\!-\!t} \mathbb{E}[\|Y^\dagger(t)\|],
    \end{align*}
    where $\beta_4, \beta_5>0$ are appropriate constants and $\beta_1$ is as defined in Lemma~\ref{lem:H}.
\end{lemma}

\begin{proof}
The proof follows a similar approach, utilizing the triangle inequality, as in:
\begin{align*}
    &\|X(k\!+\!1)\!-\!X(k)\| \leq  \|X^\dagger(k\!+\!1)\|+\|X^\dagger(k)\| \\ & \ \ \ \ \ \ \ 
 \ \ \ \ \ \ \ \ \ \ \ \ \  +\|\mathbf{1}\bar{\bx}^\intercal\!(k\!+\!1)\!-\!{\mathbf{1} \bx^*}^\intercal\!\| + \|\mathbf{1}\bar{\bx}^\intercal\!(k)\!-\!{\mathbf{1} \bx^*}^\intercal\!\|\\
    &\leq \sigma_2^{k+1} \|X^\dagger(0)\| + \alpha \sum\nolimits_{t=0}^{k} \sigma_2^{k-t}\|Y^\dagger(t)\|+\sigma_2^{k} \|X^\dagger(0)\|\\
    &\ \ \ \ + \alpha \sum_{t=0}^{k-1}\sigma_2^{k-1-t}\|Y^\dagger(t)\|+\|\mathbf{1}\bar{\bx}^\intercal(k+1)-{\mathbf{1} \bx^*}^\intercal\|_1 \\ &\ \ \ \ +\|\mathbf{1}\bar{\bx}^\intercal\!(\!k\!)\!-\!{\mathbf{1} \bx^*}^\intercal\|_1,
\end{align*}
where the last inequality follows from Lemma~\ref{lem:X-dagger}. Rearranging the terms and applying Lemma~\ref{lem:app1} yields:
\begin{align*}
    &\|X(k+1)\!-X(k)\| \leq 2\sigma_2^{k} \|X^\dagger(0)\|+\alpha\|Y^\dagger(k)\| \\
    & \ \ + 2\alpha \sum\nolimits_{t=0}^{k-1} \sigma_2^{k-1-t}\|Y^\dagger(t)\| +2n\sqrt{d}\beta_3 \theta^{k}+\frac{2n\sqrt{d}\beta_1}{1-\theta}\\
    & \ \ +\frac{n \sqrt{d}\alpha \beta_2 \theta}{\sigma_2\left(\theta-\sigma_2\right)}\sum_{l=0}^{k-1}\sigma_2^{k-l}\|Y^\dagger(l)\|\\
    &\leq 2\sigma_2^{k} \|X^\dagger(0)\|+\alpha\|Y^\dagger(k)\|+2\alpha \sum\nolimits_{t=0}^{k-1} \sigma_2^{k-1-t}\|Y^\dagger(t)\|\\
    &\ \ \ +2n\sqrt{d}\beta_3 \theta^{k}+\frac{2n\sqrt{d}\beta_1}{1-\theta}+\frac{n \sqrt{d}\alpha \beta_2 \theta}{\left(\theta-\sigma_2\right)}\|Y^\dagger(k-1)\|\\
    &\ \ \ +\frac{2n \sqrt{d}\alpha \beta_2 \theta}{\sigma_2 \left(\theta-\sigma_2\right)} \sum\nolimits_{l=0}^{k-2}\sigma_2^{k-1-l}\|Y^\dagger(l)\|,
\end{align*}

\noindent which can be further simplified as:
\begin{align*}
    &\|X(k\!+\!1)\!-\!X(k)\| \leq \theta^k\left(2\|X^\dagger(0)\|+2n\sqrt{d}\beta_3\right)\\
    & \ \ \ \ \ +\alpha\|Y^\dagger(k)\| +2\alpha \sum_{t=0}^{k-1} \theta^{k-1-t}\|Y^\dagger(t)\| +\frac{2n\sqrt{d}\beta_1}{1-\theta}\\
    & \ \ \ \ \ +\frac{2n \sqrt{d}\alpha \beta_2 \theta}{\sigma_2 \left(\theta-\sigma_2\right)} \sum_{l=0}^{k-1}\theta^{k-1-l}\|Y^\dagger(l)\| \\
    &\leq \beta_4 \theta^k + \alpha \|Y^\dagger(k)\| + \frac{2n\beta_1\sqrt{d}}{1-\theta} + \beta_5 \sum_{t=0}^{k-1} \theta^{k-t} \|Y^\dagger(t)\|.
\end{align*}
Taking expectation on both sides completes the proof.
\end{proof}

\begin{lemma}\label{lem:app3}
    Let Assumptions~\ref{ass:grad}-\ref{ass:smoothness} hold. Let $\{\bx_i(k), \by_i(k)\}$ be the sequence generated by Algorithm~\ref{alg:BDASG} for all $i\in\cV$, and let the conditions in Lemma~\ref{lem:app2} hold. Then, the vector-valued function $H(\cdot)$ defined in~\eqref{eq:alg} satisfies the following inequality:
    \begin{align*}
        \mathbb{E}[\| H(X(k\!+\!1))\!-\!H(X(k)) \|] \leq \beta_7 \theta^k + \bar{L} d\alpha \mathbb{E}[\|Y^\dagger(k)\|] \\
        + \beta_6 + \eta \sum_{t=0}^{k-1} \theta^{k-1-t} \mathbb{E}[\|Y^\dagger(t)\|],  
    \end{align*}
    where $\beta_6, \beta_7$, and $\eta>0$ are appropriate constants.
\end{lemma}
\begin{proof}
    The proof follows directly from Lemmas~\ref{lem:H}-\ref{lem:app2}:
    \begin{align*}
        \|&H(\!X(k\!+\!1))\!-\!H(\!X(k)) \| \leq d\bar{L} \|X(k\!+\!1)\!-\!X(k) \|\!+\!2C\sqrt{d} \\
        & \ \ \ \ \ \leq\underbrace{\bar{L} d\beta_4}_{\coloneqq\beta_7} \theta^k + \bar{L} d\alpha \|Y^\dagger(k)\| + \underbrace{\frac{2n\bar{L} d\beta_1\sqrt{d}}{1-\theta}+ 2\sqrt{d} C}_{\coloneqq\beta_6} \\
        & \ \ \ \ \ \ \ \ + \underbrace{\bar{L}d\beta_5}_{\coloneqq\eta} \sum_{t=0}^{k-1}\theta^{k-1-t} \|Y^\dagger(t)\|,
    \end{align*}
    where the last inequality follows from Lemma~\ref{lem:app2}.
\end{proof}

We now proceed to discuss the linear rate of the gradient
 reduction, i.e., \eqref{eq:y_grad_red} can be achieved by Algorithm~\ref{alg:BDASG}. This will complete our analysis for the linear convergence.
\begin{theorem}\label{thm:grad-red}
    Let Assumptions~\ref{ass:grad}-\ref{ass:smoothness} hold. Let $\{\bx_i(k), \by_i(k)\}$ be the sequence generated by Algorithm~\ref{alg:BDASG} for all $i\in\cV$. In addition, given some constant $\gamma\in\left[\sqrt{1-\frac{\mu^2}{\bar{L}^2}},1\right)$ and a tuning parameter $\tau>0$, let the step size satisfy the condition in Theorem~\ref{thm:avg-red}. Then there exist constants $B, D>0$ such that the sequence $\{\by_i(k)\}$ satisfies $\bE[\|\bY^\dagger(k)\|]\leq D\gamma^k + B$.
\end{theorem}
\begin{proof}
    Observe that the step size $\alpha$ satisfies the bound in Theorem~\ref{thm:avg-red}, implying $\theta \in (\sigma_2, 1)$. In addition, it can be shown that $\sigma_2 + \bar{L} d \alpha < \theta < 1$ with some proper choice of $\tau$. From the definition of $\bY^\dagger$, it follows that:
    \begin{align}\label{eq:thm2-1}
        &\bE\left[ \|Y^{\dagger} (k+1) \| \right] \leq \sigma_2 \mathbb{E} \left[ \| Y^{\dagger} (k) \| \right] \nonumber\\
        &\ \ \ \ \ \ \ \ \ \ \ \ \ \ \ \ \ \ \ \ \ \ \ \ \ \ + \mathbb{E} \left[ \| H(X(k+1)) - H(X(k)) \| \right]\nonumber\\
        &\stackrel{\text{(Lemma~\ref{lem:app3})}}{\!\!\!\!\!\!\!\!\!\!\!\!\!\!\!\!\!\!\!\!\!\!\!\!\!\!\!\!\leq}\!\!\!\!\!\!\!\!\!\!\!\!\!\!\!\!\!\!\!(\sigma_2\!+\!\bar{L}d\alpha)\mathbb{E}\!\left[\| Y^{\dagger}\!(k)\!\|\right]\!\!+\!\beta_7\theta^k\!+\!\beta_6\!+\!\eta\!\!\sum_{t=0}^{k-1}\!\theta^{k\!-\!t} \mathbb{E}\!\left[\|\!Y^{\dagger}\!(t)\!\|\right]\nonumber \\
        & \leq (\sigma_2 + \bar{L}d\alpha)^{k+1} \mathbb{E}\left[\|Y^{\dagger}(0)\|\right] + \frac{\beta_8}{1\!-\!\sigma_2\!-\!\bar{L}\alpha} \nonumber\\
        &\ \ \ + \eta\sum_{t=0}^{k}(\sigma_2+\bar{L}\alpha)^{k-t}\sum_{\ell=0}^{t-1}\theta^{t-\ell} \mathbb{E}\left[\|Y^{\dagger}(\ell)\|\right]\nonumber\\
        & \leq (\sigma_2 + \bar{L}d\alpha)^{k\!+\!1} \mathbb{E}\!\left[ \| Y^{\dagger} (0) \| \right] + \frac{\beta_8}{1\!-\!\sigma_2\!-\!\bar{L}d\alpha}\nonumber\\
        &\ \ \ +\frac{\eta\theta}{\theta\!-\!\sigma_2\!-\!\bar{L}d\alpha} \sum_{t=0}^{k-1}\!\mathbb{E}\!\left[ \| Y^{\dagger} (t) \| \right]\!\theta^{k-t},
    \end{align}
    where the last inequality follows from:

    \begin{align*}
        &\sum_{t=0}^{k}(\sigma_2 + \bar{L}d\alpha)^{k-t} \sum_{\ell=0}^{t-1} \theta^{t-\ell} \mathbb{E} \left[ \| Y^{\dagger} (\ell) \| \right] \\
        &= (\sigma_2 + \bar{L}d\alpha)^k \sum_{\ell=0}^{k-1} \mathbb{E} \left[ \| Y^{\dagger} (\ell) \| \right] \theta^{-\ell}\sum_{t=\ell+1}^{k} \left( \frac{\theta}{\sigma_2 + \bar{L}d\alpha} \right)^t \notag \\
        &\leq \frac{\theta}{\theta - \sigma_2 - \bar{L}d\alpha} \sum_{\ell=0}^{k-1} \mathbb{E} \left[ \| Y^{\dagger} (\ell) \| \right] \theta^{k-\ell} \notag
    \end{align*}

Let $r(k-1) = \sum_{t=0}^{k-1} \theta^{-t} \mathbb{E} \left[ \| Y^{\dagger} (t) \| \right]$. Then $r(k)$ is a non-decreasing, non-negative function with $r(-1) = 0$. In addition, from \eqref{eq:thm2-1} and using a proper choice of constants $\beta_9, \beta_{10}$ we have
\begin{align}\label{eq:yk1}
    \mathbb{E} \left[ \| Y^{\dagger} (k+1) \| \right] 
    \leq &\beta_9^{k+1} \mathbb{E} \left[ \| Y^{\dagger} (0) \| \right] + \beta_{10} \nonumber \\
    &+ \frac{\eta}{\theta\!-\!\sigma_2\!-\!\bar{L}d\alpha} \theta^{k+1} r(k-1).
\end{align}
We now provide an upper bound for $r(k)$. Consider:
\begin{align}
    &r(k) - r(k\!-\!1) = \theta^{-k} \mathbb{E} \left[ \| Y^{\dagger} (k) \| \right] \notag \\
    &\leq \theta^{-k}\left(\!\beta_9^k \mathbb{E} \left[ \| Y^{\dagger} (0) \| \right] + \beta_{10} + \frac{\eta}{\theta\!-\!\sigma_2\!-\!\bar{L}d\alpha} \theta^k r(k\!-\!2) \!\right) \notag \\
    &\leq \mathbb{E} \left[ \| Y^{\dagger} (0) \| \right] + \beta_{10} \theta^{-k} + \frac{\eta}{\theta - \sigma_2 - \bar{L}d\alpha} r(k-2), \notag
\end{align}
which since $r(-1) = 0$, can be simplified to:
\begin{align}\label{eq:r}
    r(k) &\leq\mathbb{E} \left[\!\| Y^{\dagger}(0) \|\!\right]\!+\!\beta_{10}\theta^{-k}\!+\! \left(\!\frac{\eta}{\theta\!-\!\sigma_2\!-\!\bar{L}d\alpha}\!+\!1\!\right)\!r(k\!-\!1) \notag \\
    &\leq \mathbb{E} \left[ \| Y^{\dagger} (0) \| \right] \!\sum\nolimits_{t=0}^{k-1}\!\left( 1 + \frac{\eta}{\theta\!-\!\sigma_2\!-\!\bar{L}d\alpha} \right)^t \notag\\
    & \ \ \ \ +\beta_{10} \sum\nolimits_{t=0}^{k-1} \theta^{-k+t}\left( 1 + \frac{\eta}{\theta\!-\!\sigma_2\!-\!\bar{L}d\alpha} \right)^t.
\end{align}

For the given bounds on $\alpha$, one can show that
\begin{equation}\label{eq:theta_times}
    \theta\left(1+\frac{\eta}{\theta-\sigma_2-\bar{L}d\alpha}\right) \leq 1.
\end{equation}

Thus, from \eqref{eq:r} and \eqref{eq:theta_times}, it follows that:
\begin{equation}\label{eq:rf}
    r(k) \leq \frac{\theta^{-k}}{1-\theta}\mathbb{E}[\|Y^\dagger(0)\|] + \beta_{11} \theta^{-k}.
\end{equation}

Substituting \eqref{eq:rf} into \eqref{eq:yk1} yields:
\begin{align*}
    \mathbb{E}[\|Y^\dagger(k+1)\|] &\leq \gamma^{k+1}\mathbb{E}[\|Y^\dagger(0)\|]+\beta_{10}\\
    &+\frac{\eta\theta^2}{\theta - \sigma_2 - \bar{L}d\alpha}\left(\frac{\mathbb{E}[\|Y^\dagger(0)\|]}{1-\theta}+\beta_{11}\right),
\end{align*}
which completes the linear convergence analysis.
\end{proof}

\section{NUMERICAL EXPERIMENTS}\label{sec:exp}
We evaluate the performance of the proposed BDASG algorithm on two problem settings: (i) the Distributed Sensor Network problem, and (ii) the Distributed Linear Regression problem with rank-deficient data. The first experiment is conducted over 150-node randomly connected network, while for the second experiment two distinct network topologies are considered: a ring graph and a star graph. For both cases, we employ the Metropolis weighting scheme to construct the adjacency matrix of the communication graph.  

\[
a_{ij} =
\begin{cases} 
\frac{1}{1 + \max\{d_i, d_j\}} & \text{if } i \neq j, \\[10pt]
1 - \sum_{j \neq i} a_{ij} & \text{if } i = j
\end{cases},
\]
where \(d_i\) is the number of neighbors for agent \(i\).

\noindent \textbf{Distributed Sensor Network Problems:}  
The problem is formulated as the following unconstrained optimization task:  
\[
\min_{\bx \in \mathbb{R}^d} \sum_{i=1}^{n} \| \bz_i - H_i \bx \|_2^2,
\]
where \(H_i \in \mathbb{R}^{m \times d}\) represents the measurement matrix, and \(\bz_i \in \mathbb{R}^m\) denotes the noisy observation at the \(i\)-th sensor. The secondary objective function is given by  
\[
\sum_{i=1}^{n} \|\bx\|^2.
\]  

In our implementation, we set the regularization parameter to \(\lambda = 0.01\) and the learning rate to \(\gamma = 0.01\). Stochastic gradients are employed in the simulation, incorporating noise sampled from a normal distribution with mean 0 and standard deviation 0.01. The matrices \(H_i\) and vectors \(\bz_i\) are randomly generated with parameters \(n = 150\), \(d = 30\), and \(m = 1\). The experiment is repeated 50 times using the same \(H_i\) and \(\bz_i\), and we report the average results (Figure~\ref{fig:fig1}b) to assess the performance of the algorithm in expectation.  we evaluate the algorithm’s performance using the error metric \( \| \bar{\mathbf{x}} - \bxs \| \), where \( \bar{\mathbf{x}} \) represents the average of the agents' local estimates, and \( \bxs \) is the centralized optimal solution.

\begin{figure*}[!ht]
	\begin{center}
		\begin{tabular}{cc}   \includegraphics[width=0.7\columnwidth]{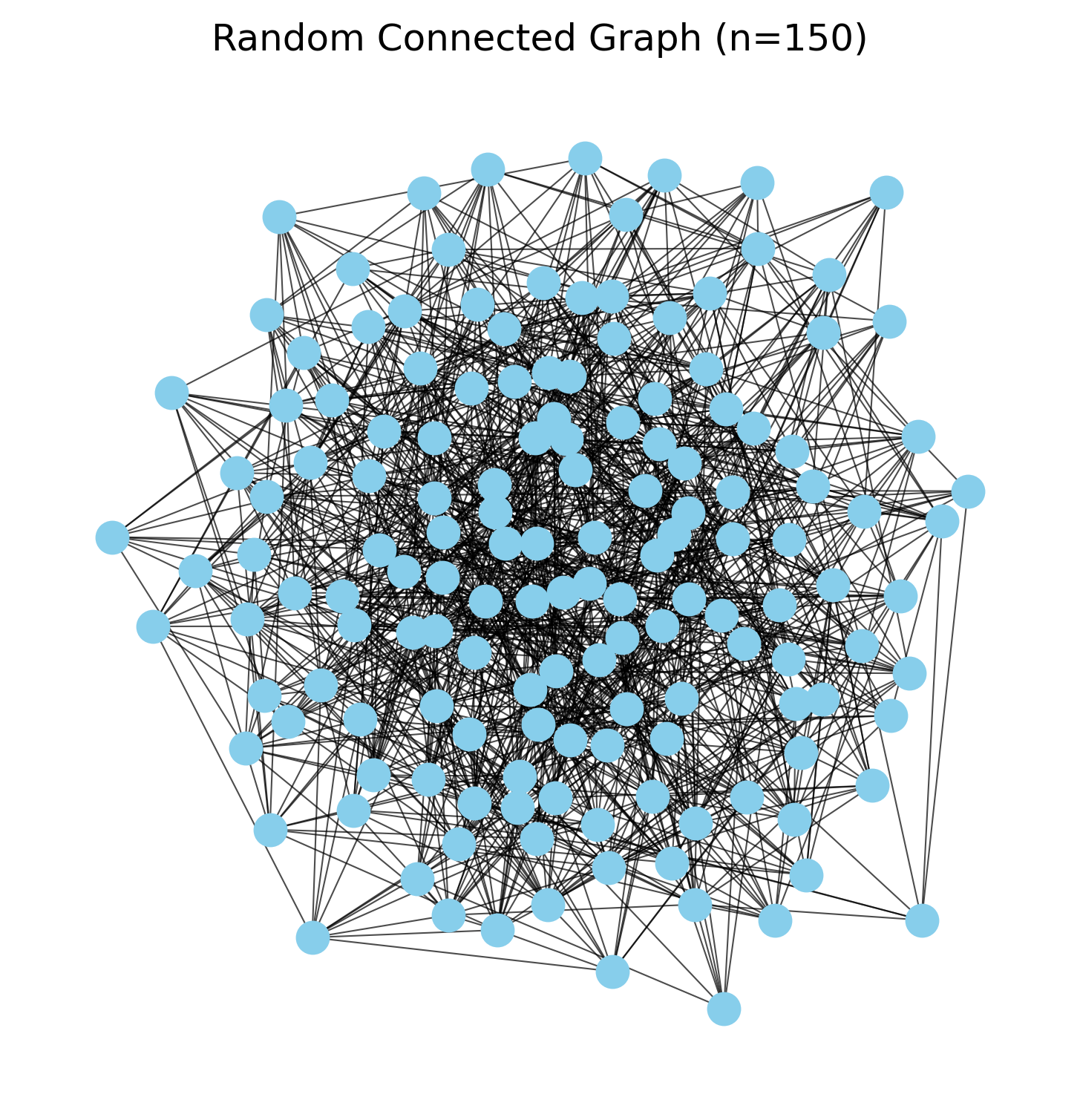} & \includegraphics[width=0.9\columnwidth]{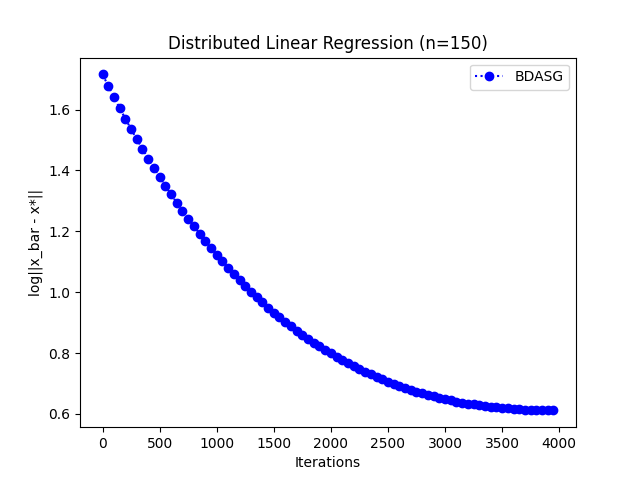} \cr
                	(a) & (b)\cr
                	\includegraphics[width=0.85\columnwidth]{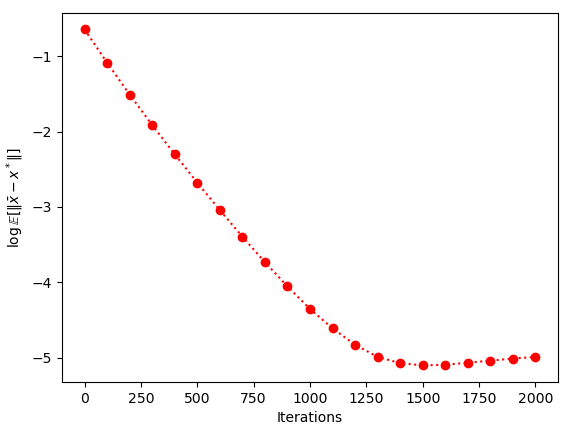} & \includegraphics[width=0.85\columnwidth]{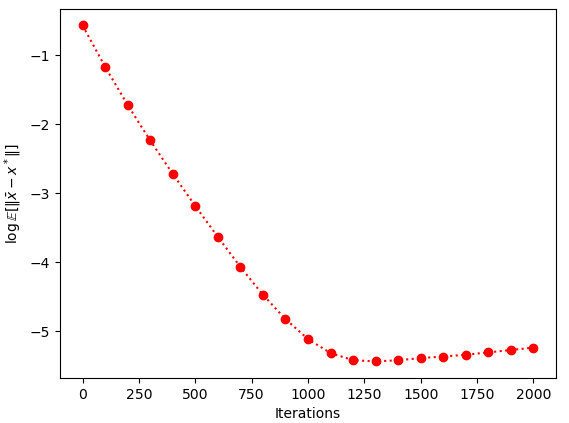}\cr
    			(c) & (d)
		\end{tabular}
	\end{center}
    \vspace{-1em}
        \caption{Plots for Distributed sensor network problem on 150-node network -- (a) network topology, (b) convergence behavior. Plots for distributed linear regression with rank deficiency problem on (c) star graph, and (d) ring graph, respectively.}
        \label{fig:fig1}
        \vspace{-.5em}
\end{figure*}


\noindent \textbf{Distributed Linear Regression with Rank Deficiency:} We assess the performance of the BDASG algorithm in solving a Lasso regression problem. The objective is to solve the following regularized ill-posed problem:  
\[
\min_{\bx \in \mathbb{R}^d} \sum_{i=1}^{n} \| A_i \bx - \mathbf{b}_i \|_2^2 + \lambda \|\bx\|_1,
\]
where \( A_i \in \mathbb{R}^{m \times d} \) represents the local measurement matrix of the \(i\)-th agent, and \( \mathbf{b}_i \in \mathbb{R}^{m} \) is the corresponding observation. To introduce ill-conditioning, one agent is assigned a rank-deficient matrix \(A_i\).  For our implementation, we set the regularization parameter to \( \lambda = 0.1 \) and the learning rate to \( \gamma = 0.001 \). The network topology is modeled using both a ring graph and a star graph, with the adjacency matrix constructed using the Metropolis weighting scheme. Once again, stochastic gradients are utilized, with noise sampled from a normal distribution with mean 0 and standard deviation 0.001.  

We generate \( A_i \) and \( b_i \) randomly for \( n = 9 \) agents, with \( d = 3 \) and \( m = 3 \). The experiment runs for 2000 iterations, and we evaluate the algorithm’s performance using the error metric \( \| \bar{\mathbf{x}} - \bxs \| \), where \( \bar{\mathbf{x}} \) represents the average of the agents' local estimates, and \( \bxs \) is the centralized optimal solution obtained via the proximal gradient method. The results are presented in Figures~\ref{fig:fig1}c and d.


\section{CONCLUSIONS}
In this work, we address the problem of bilevel distributed optimization by proposing the BDASG algorithm, which achieves linear convergence under minimal assumptions. Our theoretical analysis establishes that BDASG ensures convergence when the global objective function is strongly convex, without requiring convexity of individual functions. Furthermore, we conjecture that the algorithm remains effective under the weaker PL condition, suggesting broader applicability. Our findings have significant implications for large-scale distributed optimization, particularly in scenarios where individual objective functions lack strong convexity.


\addtolength{\textheight}{-12cm}   







\bibliographystyle{IEEEtran}
\bibliography{autosam}

@inproceedings{doan2018aggregating,
  title={Aggregating stochastic gradients in distributed optimization},
  author={Doan, Thinh T},
  booktitle={2018 Annual American Control Conference (ACC)},
  pages={2170--2175},
  year={2018},
  organization={IEEE}
}

@inproceedings{yousefian2021bilevel,
  author={Yousefian, Farzad},
  booktitle={2021 American Control Conference (ACC)}, 
  title={Bilevel Distributed Optimization in Directed Networks}, 
  year={2021},
  volume={},
  number={},
  pages={2230-2235},
  doi={10.23919/ACC50511.2021.9483429}
}

@misc{Vu2017ConvexOptimization,
  author       = {Trung Vu},
  title        = {Convex Optimization Notes},
  year         = {2017},
  month        = {April},
  url          ={https://trungvietvu.github.io/notes/2017/ConvexOptimization},
  note         = {Accessed: February 5, 2025}
}

@article{hu2016smooth,
  title={Smooth finite-time fault-tolerant attitude tracking control for rigid spacecraft},
  author={Hu, Qinglei and Shao, Xiaodong},
  journal={Aerospace Science and Technology},
  volume={55},
  pages={144--157},
  year={2016},
  publisher={Elsevier}
}

@inproceedings{nathan2017optimization,
  title={Optimization for large-scale machine learning with distributed features and observations},
  author={Nathan, Alexandros and Klabjan, Diego},
  booktitle={Machine Learning and Data Mining in Pattern Recognition: 13th International Conference, MLDM 2017, New York, NY, USA, July 15-20, 2017, Proceedings 13},
  pages={132--146},
  year={2017},
  organization={Springer}
}

@phdthesis{Tsitsiklis1984,
  title={Problems in decentralized decision making and computation},
  author={Tsitsiklis, John N},
  year={1984},
  school={Massachusetts Institute of Technology}
}

@article{Nedic2009,
  title={Subgradient methods for saddle-point problems},
  author={Nedi{\'c}, Angelia and Ozdaglar, Asuman},
  journal={Journal of optimization theory and applications},
  volume={142},
  pages={205--228},
  year={2009},
  publisher={Springer}
}

@article{Shi2015,
  title={EXTRA: An exact first-order algorithm for decentralized consensus optimization},
  author={Shi, Wei and Ling, Qing and Wu, Gang and Yin, Wotao},
  journal={SIAM Journal on Optimization},
  volume={25},
  number={2},
  pages={944--966},
  year={2015},
  publisher={SIAM}
}

@article{pu2021distributed,
  title={Distributed stochastic gradient tracking methods},
  author={Pu, Shi and Nedi{\'c}, Angelia},
  journal={Mathematical Programming},
  volume={187},
  number={1},
  pages={409--457},
  year={2021},
  publisher={Springer}
}

@article{Boyd2011,
  title={Distributed optimization and statistical learning via the alternating direction method of multipliers},
  author={Boyd, Stephen and Parikh, Neal and Chu, Eric and Peleato, Borja and Eckstein, Jonathan and others},
  journal={Foundations and Trends{\textregistered} in Machine learning},
  volume={3},
  number={1},
  pages={1--122},
  year={2011},
  publisher={Now Publishers, Inc.}
}

@article{Scutari2020,
  title={Distributed nonconvex constrained optimization over time-varying digraphs},
  author={Scutari, Gesualdo and Sun, Ying},
  journal={Mathematical Programming},
  volume={176},
  pages={497--544},
  year={2019},
  publisher={Springer}
}

@article{pu2020push,
  title={Push--pull gradient methods for distributed optimization in networks},
  author={Pu, Shi and Shi, Wei and Xu, Jinming and Nedi{\'c}, Angelia},
  journal={IEEE Transactions on Automatic Control},
  volume={66},
  number={1},
  pages={1--16},
  year={2020},
  publisher={IEEE}
}

@article{han2024distributed,
  title={Distributed adaptive gradient algorithm with gradient tracking for stochastic non-convex optimization},
  author={Han, Dongyu and Liu, Kun and Lin, Yeming and Xia, Yuanqing},
  journal={IEEE Transactions on Automatic Control},
  year={2024},
  publisher={IEEE}
}

\end{document}